\documentclass[12pt,reqno]{amsart}%
\usepackage{amsfonts}
\usepackage{geometry}
\usepackage{amsmath}
\usepackage{amssymb}
\usepackage{amscd}
\usepackage{xspace}
\usepackage{fancyhdr}
\usepackage{graphicx}
\usepackage[pagebackref]{hyperref}
\providecommand{\U}[1]{\protect\rule{.1in}{.1in}}
\hypersetup{colorlinks,breaklinks,linkcolor={red},citecolor={blue},urlcolor={green}}
\newtheorem{theorem}{Theorem}[section]
\theoremstyle{plain}

\newtheorem{lemma}{Lemma}[section]

\newtheorem{remark}{Remark}[section]

\numberwithin{equation}{section}
\begin{document}
\title{Characterizations of anisotropic high order Sobolev spaces}
\author[Nguyen Lam, Ali Maalaoui, Andrea Pinamonti]{Nguyen Lam$^{1}$, Ali Maalaoui$^{2}$, Andrea Pinamonti$^{3}$}
\date{\today}
\addtocounter{footnote}{1}
\footnotetext{University of British Columbia, Department of Mathematics
and The Pacific Institute for the Mathematical Sciences
Vancouver, BC V6T1Z4, Canada; E-mail address: {\tt{ nlam@math.ubc.ca}}}
\addtocounter{footnote}{1}
\footnotetext{Department of mathematics and natural sciences, American University of Ras Al Khaimah, PO Box 10021, Ras Al Khaimah, UAE. E-mail address:
{\tt{ali.maalaoui@aurak.ac.ae}}}
\addtocounter{footnote}{1}
\footnotetext{Dipartimento di Matematica, Universita di Trento, Via Sommarive 14, 38123 Povo, Trento, Italy. E-mail address:
{\tt{andrea.pinamonti@unitn.it}}}
\begin{abstract}
We establish two types of characterizations for high order anisotropic
Sobolev spaces. In particular, we prove high order anisotropic versions of Bourgain-Brezis-Mironescu's formula and 
Nguyen's formula.

\end{abstract}
\subjclass[2010]{ }
\maketitle

\section{Introduction}
The celebrated Bourgain-Brezis-Mironescu formula, appeared for the first time
in \cite{bourg,bourg2}, and provided a new characterization for functions in the Sobolev space $W^{1,p}(\mathbb{R}^N)$, with $1<p<\infty$.
More precisely, they proved\\

\textbf{Theorem A. (Bourgain, Brezis and Mironescu, \cite{bourg}).}
\textit{Let }$g\in L^{p}\left(
\mathbb{R}
^{N}\right)  ,~1<p<\infty.$\textit{ Then }$g\in W^{1,p}\left(
\mathbb{R}
^{N}\right)  $\textit{ iff}%
\[
\underset{%
\mathbb{R}
^{N}}{%
{\displaystyle\int}
}\underset{%
\mathbb{R}
^{N}}{%
{\displaystyle\int}
}\frac{\left\vert g(x)-g(y)\right\vert ^{p}}{\left\vert x-y\right\vert ^{p}%
}\rho_{n}\left(  \left\vert x-y\right\vert \right)  dxdy\leq C,~\forall
n\geq1,
\]
\textit{for some constant }$C>0.$\textit{ Moreover, }%
\[
\underset{n\rightarrow\infty}{\lim}\underset{%
\mathbb{R}
^{N}}{%
{\displaystyle\int}
}\underset{%
\mathbb{R}
^{N}}{%
{\displaystyle\int}
}\frac{\left\vert g(x)-g(y)\right\vert ^{p}}{\left\vert x-y\right\vert ^{p}%
}\rho_{n}\left(  \left\vert x-y\right\vert \right)  dxdy=K_{N,p}%
{\displaystyle\int\limits_{\mathbb{R}^{N}}}
\left\vert \nabla g(x)\right\vert ^{p}dx.
\]
\textit{Here }%
\begin{align}\label{defK}
K_{N,p}=%
{\displaystyle\int\limits_{\mathbb{S}^{N-1}}}
\left\vert e\cdot\sigma\right\vert ^{p}d\sigma
\end{align}
\textit{for any }$e\in\mathbb{S}^{N-1}$ \textit{and} $d\sigma$ \textit{is the
surface measure on} $\mathbb{S}^{N-1}$. \textit{ Here }$\left(  \rho
_{n}\right)  _{n\in%
\mathbb{N}
}$\textit{ is a sequence of nonnegative radial mollifiers satisfying}%
\begin{align*}
\underset{n\rightarrow\infty}{\lim}%
{\displaystyle\int\limits_{\tau}^{\infty}}
\rho_{n}\left(  r\right)  r^{N-1}dr  &  =0\quad \forall\tau>0,\qquad
\underset{n\rightarrow\infty}{\lim}%
{\displaystyle\int\limits_{0}^{\infty}}
\rho_{n}\left(  r\right)  r^{N-1}dr =1.
\end{align*}

Starting from the previous result and since the theory of Sobolev spaces is a fundamental tool in many branches of modern mathematics, such as harmonic analysis,
complex analysis, differential geometry and geometric analysis, partial
differential equations, etc, there has been a substantial effort to
characterize Sobolev spaces in different settings  (see e.g.,\cite{AmbDepMart},
\cite{Barb}, \cite{BHN3},\cite{bre},\cite{bre-linc},\cite{Cui}, \cite{Cui2}, \cite{KL}, \cite{NguPinSquVec2},\cite{PinSquVec}, \cite{PinSquVec2}, \cite{BM},\cite{Ponce}). 

Theorem A has been extended to the high order case by Bojarski, Ihnatsyeva and
Kinnunen \cite{BIK2011} using the high order Taylor remainder and by Borghol
\cite{B} using high order differences.

We note here, as a consequence of Theorem A, that we can characterize the
Sobolev space $W^{1,p}(\mathbb{R}^{N})$ as follows: Let $g\in L^{p}\left(
\mathbb{R}
^{N}\right)  ,~1<p<\infty.$ Then $g\in W^{1,p}\left(
\mathbb{R}
^{N}\right)  $ iff%
\begin{equation}
\sup_{0<\delta<1}\underset{\left\vert x-y\right\vert <\delta}{\underset{%
\mathbb{R}
^{N}}{%
{\displaystyle\int}
}\underset{%
\mathbb{R}
^{N}}{%
{\displaystyle\int}
}}\frac{\left\vert g(x)-g(y)\right\vert ^{p}}{\delta^{N+p}}dxdy<\infty.
\label{eq1}%
\end{equation}

Recently, Nguyen \cite{nguyen06} (see also \cite{nguyen07}), motivated by an estimate for the topological degree for the Gizburg-Landau equation (\cite{BBN23}), established some new characterizations of the
Sobolev space $W^{1,p}(\mathbb{R}^{N})$ which are closely related to Theorem
A. More precisely, he used the dual form of (\ref{eq1}) and proved the following
results:\newline

\medskip

\textbf{Theorem B. (H. M. Nguyen, \cite{nguyen06}). }\textit{Let }$1<p<\infty
.$\textit{ Then the following hold:}

\textit{(a) Let $g\in W^{1, p}(\mathbb{R} ^{N})$. Then there exists a positive
constant }$C_{N,p}$\textit{ depending only on N and p such that }%
\[
\underset{\left\vert g(x)-g(y)\right\vert >\delta}{%
{\displaystyle\int\limits_{\mathbb{R} ^{N}}}
{\displaystyle\int\limits_{\mathbb{R} ^{N}}}
}\frac{\delta^{p}}{\left\vert x-y\right\vert ^{N+p}}dxdy\leq C_{N,p}%
{\displaystyle\int\limits_{\mathbb{R} ^{N}}}
\left\vert \nabla g(x)\right\vert ^{p}dx,~\forall\delta>0,\forall g\in
W^{1,p}\left(
\mathbb{R}
^{N}\right)  .
\]

\textit{(b) If }$g\in L^{p}\left(
\mathbb{R}
^{N}\right)  $\textit{ satisfies }%
\[
\underset{0<\delta<1}{\sup}\underset{\left\vert g(x)-g(y)\right\vert >\delta}{%
{\displaystyle\int\limits_{\mathbb{R} ^{N}}}
{\displaystyle\int\limits_{\mathbb{R} ^{N}}}
}\frac{\delta^{p}}{\left\vert x-y\right\vert ^{N+p}}dxdy<\infty,
\]
\textit{then }$g\in W^{1,p}\left(
\mathbb{R}
^{N}\right)  .$

\textit{(c) For any } $g\in W^{1,p}(\mathbb{R}^{N}),$%
\[
\underset{\delta\rightarrow0}{\lim}\underset{\left\vert g(x)-g(y)\right\vert
>\delta}{%
{\displaystyle\int\limits_{\mathbb{R}^{N}}}
{\displaystyle\int\limits_{\mathbb{R}^{N}}}
}\frac{\delta^{p}}{\left| x-y\right|^{N+p}}dxdy=\frac{1}{p}K_{N,p}%
{\displaystyle\int\limits_{\mathbb{R}^{N}}}
\left\vert \nabla g(x)\right\vert ^{p}dx,
\]
where $K_{N,p}$ is as in \eqref{defK}.

The previous result has been generalized in many ways and for different spaces (see e.g. \cite{Cui,NguPinSquVec,NguSq,MSqu}). We recall in particular the folowing result proved in \cite{NguSq}\\

\textbf{Theorem C. (H. M. Nguyen, M. Squassina \cite{NguSq}). }
Let $1<p<\infty$ and $K\subset\mathbb{R}^N$ be a convex, symmetric set containing the origin and with nonempty interior. Then, for every $g\in W^{1,p}_K(\mathbb{R}^N)$,
\[
\lim_{\delta\to 0}\underset{\left\vert g(x)-g(y)\right\vert >\delta}{%
{\displaystyle\int\limits_{\mathbb{R} ^{N}}}
{\displaystyle\int\limits_{\mathbb{R} ^{N}}}
}\frac{\delta^{p}}{\| x-y\|_K ^{N+p}}dxdy=\int\limits_{\mathbb{R}^N} \|\nabla g\|_{Z^{*}_p K}^p\, dx,
\]
where $\|\cdot \|_{K}$ is the norm in $\mathbb{R}^{N}$ which admits
as unit ball the set $K$, i.e. $\|x\|_{K}:=\inf\{\lambda>0\ |\ \frac{x}{\lambda}\in K \}$, $\|\cdot\|_{Z^{*}_p K}$ is the norm associated with the $L_p$ polar body of $K$, namely
\begin{equation}
\|v\|_{Z^{*}_p K}=\left(\frac{N+p}{p}\int_K |v\cdot x|^p\, dx\right)^{1/p},\qquad v\in \mathbb{R}^N
\end{equation}
and $W^{1,p}_K(\mathbb{R}^N)$ is the associated Sobolev space.\\

The main purpose of this paper is to generalize Theorem A and Theorem C to high-order anisotropic Sobolev spaces. In order to describe our main results we recall the following notation (\cite{B}): 
Let $f\in W^{k,p}(\Omega)$ and
$\sigma=(\sigma_{1},\ldots, \sigma_{N})\in\mathbb{R}^{N}$, we denote 

\[
D^{k}f\left(  x\right)  \left(  \sigma,...,\sigma\right)  =%
{\displaystyle\sum\limits_{1\leq i_{1},...,i_{k}\leq N}}
\sigma_{i_{1}}...\sigma_{i_{k}}\frac{\partial^{k}f}{\partial x_{i_{1}%
}...\partial x_{i_{k}}}\left(  x\right),\qquad \mbox{a.e.}\ x\in\Omega 
\]
and, for every $m\in\mathbb{N}$
\begin{align}\label{defR}
R^{m}f(x,y)=\sum_{j=0}^{m}(-1)^{j}\binom{m}{j}f\left(  \frac{m-j}{m}x+\frac
{j}{m}y\right).
\end{align}
\begin{theorem}
\label{conb} Let $K\subset\mathbb{R}^N$ be a convex, symmetric set containing the origin and with nonempty interior. Let $f\in W^{m,p}(\mathbb{R}^{N})$ with $m\in\mathbb{N}$ and $1<p<\infty$. Then
\begin{align}\label{54}
\underset{\delta\rightarrow0}{\lim}\underset{\left\vert R^mf\left(
x,y\right)  \right\vert >\delta}{%
{\displaystyle\int\limits_{\mathbb{R}^{N}}}
{\displaystyle\int\limits_{\mathbb{R}^{N}}}
}\frac{\delta^{p}}{\left\Vert x-y\right\Vert _{K}^{N+mp}}dxdy=\frac{N+mp}{m^{mp+1}p}%
{\displaystyle\int\limits_{\mathbb{R}^{N}}}
\int\limits_{K}\left\vert D^{m}f(x)(y,...,y)\right\vert ^{p}dydx.
\end{align}
\end{theorem}
Notice that taking $m=1$ in the previous theorem we get Theorem C and taking $m=2$ and $\|\cdot\|$ as the Euclidean norm we get \cite[Theorem 1.1]{Cui2}.\\
Our next result is the analogous of \cite[Theorem 4]{B} in our setting.
\begin{theorem}\label{Bani}
Let $K\subset\mathbb{R}^N$ be a convex, symmetric set containing the origin with nonempty interior.
Let $(\rho_{\varepsilon})_{\varepsilon}$ be a family of functions $\rho_{\varepsilon}:[0,\infty)\to [0,\infty)$ satisfying the following conditions
\begin{align}\label{propq}
\int_{0}^{\infty}r^{N-1}\rho_{\varepsilon}(r)dr=1\quad\text{ and }\quad\lim_{\varepsilon\to 0}\int_{\delta}^{\infty
}r^{N-1}\rho_{\varepsilon}(r)dr=0\qquad \forall\ \delta>0
\end{align}
Let $f\in W^{m,p}(\mathbb{R}^{N})$ with $m\in\mathbb{N}$ and $1<p<\infty$, then
\begin{equation}
\lim_{\varepsilon\to0} \int_{\mathbb{R}^{N}}\int_{\mathbb{R}^{N}} \frac{|R^{m}
f(x,y)|^{p}}{\|x-y\|^{mp}_{K}} \rho_{\varepsilon}(\|x-y\|_{K})\ dxdy =
\frac{N+mp}{m^{mp}}\int_{\mathbb{R}^{N}}\int_{K}|D^{m} f(x)(y,\cdots
,y)|^{p} dy dx.
\end{equation}
\end{theorem}
Our next result results can be considered a generalization to high-order anisotropic spaces of \cite[Theorem 1.2]{Cui2}. For any $f\in W^{m,p}(\mathbb{R}^N)$ and $y\in \mathbb{R}^N$ let
\[
T_{y}^{m-1}f\left(  x\right)  =%
{\displaystyle\sum\limits_{\left\vert \alpha\right\vert \leq m-1}}
D^{\alpha}f\left(  y\right)  \frac{\left(  x-y\right)  ^{\alpha}}{\alpha!}%
\]
and
\[
R_{m-1}f\left(  x,y\right)  =f\left(  x\right)  -T_{y}^{m-1}f\left(  x\right)  .
\]
Then we will prove the following result

 \begin{theorem}\label{lam1}
Let $K\subset\mathbb{R}^N$ be a convex, symmetric set containing the origin, with nonempty interior and $f\in W^{m,p}(\mathbb{R}^{N})$ with $1<p<\infty$ and $m\in\mathbb{N}$. Then 
\[
\lim_{\delta\to0}
\underset{\left\vert R_{m-1}f\left(  x,y\right)  \right\vert >\delta}{%
{\displaystyle\int\limits_{\mathbb{R}^{N}}}
{\displaystyle\int\limits_{\mathbb{R}^{N}}}
}\frac{\delta^{p}}{\left\Vert x-y\right\Vert _{K}^{N+mp}}dxdy=\frac{N+mp}{\left(
m!\right)  ^{p}mp}%
{\displaystyle\int_{\mathbb{R}^{N}}}
\int_{K}\left\vert D^{m}f(x)(y,\cdots,y)\right\vert ^{p}dydx.
\]

\end{theorem}
\begin{theorem}\label{lam2}
Let $K\subset\mathbb{R}^N$ be a convex, symmetric set containing the origin with nonempty interior and $f\in W^{m,p}(\mathbb{R}^{N})$ with $1<p<\infty$ and $m\in\mathbb{N}$. Then 
$$ \underset{\varepsilon\rightarrow0}{\lim}\int_{\mathbb{R}^{N}}\int
_{\mathbb{R}^{N}}\frac{\left\vert R_{m-1}f\left(  x,y\right)  \right\vert
^{p}}{\left\Vert x-y\right\Vert _{K}^{mp}}\rho_{\varepsilon}\left(  \left\Vert
x-y\right\Vert _{K}\right)  dxdy\\
 =\frac{N+mp}{\left(  m!\right)  ^{p}}\int_{\mathbb{R}^{N}}\int_{K}\left\vert D^{m}f(x)(y,\cdots,y)\right\vert
^{p}dydx.
$$
Here the family $(\rho_{\varepsilon})_{\varepsilon}$ is as in Theorem \ref{Bani}.
\end{theorem}
The plan of the paper is the following: In Section 2, we will introduce some terminology. In Section 3 we will prove Theorems \ref{Bani} and \ref{conb}. Finally, in Section 4, we will establish Theorems \ref{lam1} and \ref{lam2} .

\section{Anisotropic spaces}
Let $|x|$ be the Euclidean norm of $x\in \mathbb{R}^N$ and let us fix a convex, symmetric subset $K\subset\mathbb{R}^N$ containing the origin with nonempty interior.
For a multi-index $\alpha=(\alpha_1\ldots,\alpha_N)$, $\alpha_i\geq 0$ and a point $x=(x_1,\ldots, x_N)\in \mathbb{R}^N$ we denote by
\[
x^{\alpha}=\prod_{i=1}^N x_i^{\alpha_i}\qquad \mbox{and}\qquad |\alpha|=\sum_{i=1}^N \alpha_i.
\]
In the same way 
\[
D^{\alpha}u=\frac{\partial^{|\alpha|}u}{\partial_{x_1}^{\alpha_1}\ldots \partial_{x_N}^{\alpha_N}}
\]
is a weak partial derivative of order $u$. We also denote by $\nabla^m u$ a vector with components $D^{\alpha}u$, $|\alpha|=m$.
 As in Theorem C, we denote by $\|\cdot\|_{K}$ be the norm 
$\|x\|_K=\inf\{\lambda>0\ |\ x/\lambda\in K\}$ and we observe that since all norms on $\mathbb{R}^{N}$ are equivalent, there are $A,B>0$ such that
\begin{equation}
A\left\vert \cdot\right\vert \leq\Vert\cdot\Vert_{K}\leq B\left\vert
\cdot\right\vert \label{*}
\end{equation}
Now let $y\in K$ and denote by $E$ the vector with components $E_{\alpha}=(1/\alpha!)y_1^{\alpha_1}\cdot y_N^{\alpha_n}$, $|\alpha|=m$. Then, for every $m\in\mathbb{N}$ and $1\leq p<\infty$ the function
\[
\|v\|_{Z_{p,m}^{*}K}=\left(\frac{N+mp}{m^{mp+1}p}\int_K |v\cdot E|^p dy\right)^{1/p}
\] 
is a norm on a linear space of all vectors $v=(v_{\alpha})_{|\alpha|=m}$ \cite{BIK2011}. 
Set
\[
[u]_{W^{m,p}_K(\mathbb{R}^N)}=\left(\int_{\mathbb{R}^N}\|\nabla^m u\|_{Z_{p,m}^{*}K}^p dx\right)^{\frac{1}{p}}.
\]
Fix $1\leq p<\infty$ and $m\in\mathbb{N}$, we denote by $W^{m,p}_{K}(\mathbb{R}^N)$ the space of $u\in L^p(\mathbb{R}^N)$ such that $[u]_{W^{m,p}_K(\mathbb{R}^N)}<\infty$ endowed with the norm
\[
\|u\|_{W^{m,p}_K(\mathbb{R}^N)}=\left(\|u\|_{L^p(\mathbb{R}^N)}+[u]_{W^{m,p}_K(\mathbb{R}^N)}\right)^{\frac{1}{p}}.
\]
Clearly for $m=1$ the space defined above coincides with the one studied in \cite{NguSq}, moreover taking $K=B(0,1)$ it is easy to see that $W^{m,p}_{K}(\mathbb{R}^N)=\dot{W}^{m,p}(\mathbb{R}^N)$, here $\dot{W}^{m,p}(\mathbb{R}^N)$ denotes the homogenous Sobolev space as defined in \cite[Definition 11.17]{Leoni}.\\
We will use the following notation: Given two quantities $f$ and $g$ we write $f\lesssim  g$ if there exists $C>0$ such that $f\leq Cg$.
\section{Proofs of Theorems \ref{conb} and \ref{Bani}}

Let $m\geq 1$.
Set $\Delta_{h}f(x)=\Delta_h^1f(x):=f(x+h)-f(x)$, we call m-th difference the quantity
$\Delta_{h}^{m}f(x)=\Delta_{h}\left[  \Delta_{h}^{m-1}f\right]  (x)$. By
above definition, it is not difficult to show that for any positive integer
$m$, we have
\[
\Delta_{h}^{m}f(x)=\sum_{j=0}^{m}(-1)^{m+j}\left(
\begin{array}
[c]{c}%
m\\
j
\end{array}
\right)  f(x+jh)
\]
and, by \cite[Lemma 8]{B}, we also have
\begin{equation}
\Delta_{{h}}^{m}f(x)=\int\limits_{{{[0,1]}^{m}}}{D^{m}}f\left(  x+\sum
\limits_{j=1}^{m}{{t_{j}}{h}}\right)  (h,\cdots,h)dt_{1}...dt_{m}
\label{meanva1}%
\end{equation}
where $[0,1]^m$ denotes the unit-cube in $\mathbb{R}^m$.
Finally, it is easy to see that for every $x,h\in\mathbb{R}^{N}$
\[
R^{m}f(x,x+mh)=(-1)^{m}\Delta_{h}^{m}f(x)
\]
where $R^m f(x,y)$ is as in \eqref{defR}. 
\subsection{Nguyen's formula}

The aim of this section is to prove Theorem \ref{conb}. We start with the following:

\begin{lemma}\label{pols} Let $K\subset\mathbb{R}^N$ be a convex, symmetric set containing the origin with nonempty interior and $f\in W^{m,p}(\mathbb{R}^{N})$ with $1<p<\infty$ and $m\in\mathbb{N}$.
There exists $C=C(m,N,p)>0$ s.t.
\[
\underset{\left\vert R^{m}f\left(  x,y\right)  \right\vert >\delta}{%
{\displaystyle\int\limits_{\mathbb{R}^{N}}}
{\displaystyle\int\limits_{\mathbb{R}^{N}}}
}\frac{\delta^{p}}{\Vert x-y\Vert^{N+mp}_{K}}dxdy\leq C\Vert\nabla
^{m}f\Vert_{L^{p}(\mathbb{R}^{N})}^p.%
\]

\end{lemma}

\begin{proof}
Using polar coordinates ($y=x+t\sigma$, $\sigma=\frac{y-x}{|y-x|}$ and
$t=|x-y|$) we write
\[
\underset{ \left\vert R^{m}f\left(  x,y\right)  \right\vert >\delta}{%
{\displaystyle\int\limits_{\mathbb{R}^{N}}}
{\displaystyle\int\limits_{\mathbb{R}^{N}}}
} \frac{\delta^{p}}{\Vert x-y\Vert_{K}^{N+mp}}dxdy=\int_{\mathbb{S}^{N-1}%
}\underset{|\Delta_{\frac{t}{m}\sigma}^{m}f(x)|>\delta}{%
{\displaystyle\int\limits_{\mathbb{R}^{N}}}
{\displaystyle\int\limits_{\mathbb{R}^{+}}}
}\frac{\delta^{p}}{\Vert\sigma\Vert_{K}^{N+mp}t^{1+mp}}dtdxd\sigma
\]
Thus, since $A\leq\Vert\sigma\Vert_{K}\leq B$, it is enough to show that there
exists a constant $C=C(m,N,p)>0$ such that for every $\sigma\in\mathbb{S}%
^{N-1}$
\begin{align}\label{ssx}
\underset{|\Delta_{\frac{t}{m}\sigma}^{m}f(x)|>\delta}{%
{\displaystyle\int\limits_{\mathbb{R}^{N}}}
{\displaystyle\int\limits_{\mathbb{R}^{+}}}
}\frac{\delta^{p}}{t^{1+mp}}dtdx\leq C\int_{\mathbb{R}^{N}}|\nabla
^{m}f(x)|^{p}dx.
\end{align}
We assume, without loss of generality, that $\sigma=e_{N}=(0,\ldots,0,1)$\footnote{Fix $\sigma \in \mathbb{S}^{N-1}$ and let $B\in SO(N)$ such that $B\sigma=e_N$. Then, by a change of variables and \eqref{ssx}
\begin{align*}
\int_{\mathbb{S}^{N-1}%
}\underset{|\Delta_{\frac{t}{m}\sigma}^{m}f(x)|>\delta}{%
{\displaystyle\int\limits_{\mathbb{R}^{N}}}
{\displaystyle\int\limits_{\mathbb{R}^{+}}}
}\frac{\delta^{p}}{t^{1+mp}}dtdxd\sigma&=\int_{\mathbb{S}^{N-1}%
}\underset{|\Delta_{\frac{t}{m}e_N}^{m}f(x)|>\delta}{%
{\displaystyle\int\limits_{\mathbb{R}^{N}}}
{\displaystyle\int\limits_{\mathbb{R}^{+}}}
}\frac{\delta^{p}}{t^{1+mp}}dtdxd\sigma\\
&\leq C|\mathbb{S}^{N-1}|\int_{\mathbb{R}^{N}}|\nabla^{m}f(x)|^{p}dx.
\end{align*}}. By
(\ref{meanva1}), we have for any $x=(x^{\prime},x_{N})\in\mathbb{R}%
^{N-1}\times\mathbb{R}$
\begin{align*}
|\Delta_{\frac{t}{m}e_{N}}^{m}f(x)|  &  \lesssim t^{m}\left\vert
\int\limits_{{{[0,1]}^{m}}}{\partial_{x_{N}}^{m}}f(x^{\prime},x_{N}+\frac
{t}{m}\sum\limits_{j=1}^{m}{{s_{j}}})ds_{1}...ds_{m}\right\vert \\
&  \lesssim t^{m-1}\int_{x_{N}}^{x_{N}+t}|\partial_{x_{N}}%
^{m}f(x^{\prime},s)|ds\\
&  \lesssim t^{m}\mathbb{M}_{N}\left(  \partial_{x_{N}}^{m}f\right)
(x^{\prime},x_{N})
\end{align*}
where $\mathbb{M}_{N}(f)$ denotes the maximal function of $f$ in direction
$x_{N}$, namely
\[
\mathbb{M}_{N}(f)(x)=\mathbb{M}_{N}(f)(x^{\prime},x_{N})=\sup_{t>0}\frac{1}%
{t}\int_{x_{N}}^{x_{N}+t}|f(x^{\prime},s)|ds.
\]
So, there exists $C=C(N,m,p)>0$ such that 
\begin{align*}
\underset{|\Delta_{\frac{t}{m}\sigma}^{m}f(x)|>\delta}{%
{\displaystyle\int\limits_{\mathbb{R}^{N}}}
{\displaystyle\int\limits_{\mathbb{R}^{+}}}
}\frac{\delta^{p}}{t^{1+mp}}dtdx  &  \leq\int_{\mathbb{R}^{N}}\int_{0}%
^{\infty}1_{C t^{m}\mathbb{M}_{N}\left(  \partial_{x_{N}}%
^{m}f\right)  (x)>\delta}\frac{\delta^{p}}{t^{1+mp}}dtdx\\
&  \lesssim\int_{\mathbb{R}^{N}}\mathbb{M}_{N}\left(
\partial_{x_{N}}^{m}f\right)  (x)^{p}dx\\
&  \lesssim\int_{\mathbb{R}^{N}}|\partial_{x_{N}}^{m}%
f(x)|^{p}dx\\
&  \lesssim \Vert\nabla^{m}f\Vert_{L^{p}(\mathbb{R}^{N})}^p%
\end{align*}
where in the line before the last we
used (see \cite{Stein})
\[
\int_{\mathbb{R}^{N-1}}\int_{\mathbb{R}}\left\vert \mathbb{M}_{N}\left(
\partial_{x_{N}}^{m}f\right)  (x^{\prime},x_{N})\right\vert ^{p}%
dx_{N}dx^{\prime}\lesssim\int_{\mathbb{R}^{N-1}}\int_{\mathbb{R}}\left\vert
\partial_{x_{N}}^{m}f(x^{\prime},x_{N})\right\vert ^{p}dx_{N}dx^{\prime}.
\]
This gives the conclusion.
\end{proof}
We are now in position to prove Theorem \ref{conb}.\\

\begin{proof}[\textbf{Proof of Theorem \ref{conb}:}] Notice that \eqref{54} can be re-written as:
\[
\underset{\delta\rightarrow0}{\lim}\underset{\left\vert R^mf\left(
x,y\right)  \right\vert >\delta}{%
{\displaystyle\int\limits_{\mathbb{R}^{N}}}
{\displaystyle\int\limits_{\mathbb{R}^{N}}}
}\frac{\delta^{p}}{\left\Vert x-y\right\Vert _{K}^{N+mp}}dxdy=
{\displaystyle\int\limits_{\mathbb{R}^{N}}}
\|\nabla^m f\|_{Z_{p,m}^*K}^p\, dx.
\]
By changing variables (writing $y=x+\sqrt[m]{\delta}h\sigma$, $\sigma
=\frac{y-x}{\left\vert y-x\right\vert }$, $h=\frac{1}{\sqrt[m]{\delta}%
}\left\vert y-x\right\vert $), we obtain
\[
\underset{|R^{m}f(x,y)|>\delta}{%
{\displaystyle\int\limits_{\mathbb{R}^{N}}}
{\displaystyle\int\limits_{\mathbb{R}^{N}}}
}
\frac
{\delta^{p}}{\Vert x-y\Vert^{N+mp}_{K}}dxdy=
\underset{\left\vert \frac{\Delta
_{\frac{\sqrt[m]{\delta}h\sigma}{m}}^{m}f(x)}{h^{m}\delta}\right\vert h^{m}%
>1}{
{\displaystyle\int\limits_{\mathbb{S}^{N-1}}}
{\displaystyle\int\limits_{\mathbb{R}^{N}}}
}
{\displaystyle\int_0^\infty}
\frac{1}{\Vert\sigma\Vert^{N+mp}_{K}h^{1+mp}}dhdxd\sigma
\]
By Lemma \ref{pols} there exists $C=C(m,N,p)>0$
such that for every $\sigma\in\mathbb{S}^{N-1}$,
\begin{equation}
\underset{\left\vert \frac{\Delta
_{\frac{\sqrt[m]{\delta}h\sigma}{m}}^{m}f(x)}{h^{m}\delta}\right\vert h^{m}%
>1}{
{\displaystyle\int\limits_{\mathbb{R}^{N}}}
{\displaystyle\int\limits_0^{\infty}}}
\frac{1}{h^{mp+1}}dhdx\leq C\int_{\mathbb{R}^{N}}|\nabla^{m}f(x)|^{p}dx.
\label{1}%
\end{equation}
Moreover, for every $\sigma\in\mathbb{S}^{N-1}$ it holds
\begin{equation}
\lim_{\delta\rightarrow0}\underset{\left\vert \frac{\Delta
_{\frac{\sqrt[m]{\delta}h\sigma}{m}}^{m}f(x)}{h^{m}\delta}\right\vert h^{m}%
>1}{
{\displaystyle\int\limits_{\mathbb{R}^{N}}}
{\displaystyle\int\limits_0^{\infty}}}\frac{1}{h^{mp+1}}dhdx=\frac{1}{m^{mp+1} p}%
{\displaystyle\int\limits_{\mathbb{R}^{N}}}
|D^{m}f(x)(\sigma,...,\sigma)|^{p}dx\label{23}.%
\end{equation}
To prove \eqref{23}, we define $F_{\delta
}:\mathbb{S}^{N-1}\rightarrow\mathbb{R}$ by
\[
F_{\delta}(\sigma):=\frac{1}{\left\Vert \sigma\right\Vert _{K}^{N+mp}}%
\underset{\left\vert \frac{\Delta_{\frac{\sqrt[m]{\delta}}{m}h\sigma}^{m}%
f(x)}{h^{m}\delta}\right\vert h^{m}>1}{\int_{\mathbb{R}^{N}}\int_{0}^{\infty}%
}\frac{1}{h^{mp+1}}dhdx.
\]
By Lemma \ref{pols} there exists $C=C(N,p,m)>0$ such that for all $\sigma
\in\mathbb{S}^{N-1}$ and for all $\delta>0$ :
\begin{equation}
F_{\delta}(\sigma)\leq\frac{C}{A^{N+mp}}%
{\displaystyle\int\limits_{\mathbb{R}^{N}}}
|\nabla^{m}f\left(  x\right)  |^{p}dx. \label{stima3}%
\end{equation}
Without loss of generality, we suppose that $\sigma=e_{N}$.
Given $x^{\prime}\in\mathbb{R}^{N-1}$ and
$\delta\in(0,1)$ we define
\[
\mathcal{A}(x^{\prime},\delta):=\left\{  (x_{N},h)\in\mathbb{R}\times(0,\infty
)\ |\ \left\vert \frac{\Delta_{\frac{\sqrt[m]{\delta}}{m}he_{N}}%
^{m}f(x^{\prime},x_{N})}{h^{m}\delta}\right\vert h^{m}>1\right\} ,
\]%
and
\[
\mathcal{A}(x^{\prime}):=\left\{  (x_{N},h)\in\mathbb{R}\times(0,\infty)\ |\ \left\vert
\partial_{x_{N}}^{m}f(x^{\prime},x_{N})\right\vert h^{m}>m^{m}\right\}.
\]
 Using
\eqref{meanva1} it is easy to see that
\[
\lim_{\delta\to0}\frac{1}{h^{mp+1}} 1_{\mathcal{A}(x^{\prime},\delta)}(x_{N},h)= \frac{1}{h^{mp+1}}1_{\mathcal{A}(x^{\prime})}%
(x_{N},h)\quad\mbox{a.e.}\ (x',x_{N},h)\in\mathbb{R}^{N-1}\times\mathbb{R}\times[0,\infty),
\]
and
\[
\frac{1}{h^{mp+1}} 1_{\mathcal{A}(x^{\prime},\delta)}(x_{N},h)\leq\frac{1}{h^{mp+1}}1_{\{h^{m}\mathbb{M}_{N}(\partial_{x_{N}%
}^{m}f)(x)>m^{m}\}}(x,h)\in\mathit{L}^{1}(\mathbb{R}^{N}\times\lbrack
0,\infty)).
\]
By the Lebesgue dominated convergence theorem, we get \eqref{23}.
Using \eqref{stima3} and the Lebesgue dominated convergence
theorem again, we can conclude that
\begin{align*}
\underset{\delta\rightarrow0}{\lim}\underset{\left\vert R^{m}f\left(
x,y\right)  \right\vert >\delta}{\int_{\mathbb{R}^{N}}\int_{\mathbb{R}^{N}}%
}\frac{\delta^{p}}{\left\Vert x-y\right\Vert _{K}^{N+mp}}dxdy  &
=\int_{\mathbb{S}^{N-1}}\frac{1}{m^{mp+1}p}\frac{1}{\left\Vert \sigma
\right\Vert _{K}^{N+mp}}\int_{\mathbb{R}^{N}}\left\vert D^{m}f(x)(\sigma
,...,\sigma)\right\vert ^{p}dxd\sigma\\
&  =\frac{1}{m^{mp+1}p}\int_{\mathbb{R}^{N}}\int_{\mathbb{S}^{N-1}}\frac
{1}{\left\Vert \sigma\right\Vert _{K}^{N+mp}}\left\vert D^{m}f(x)(\sigma
,...,\sigma)\right\vert ^{p}d\sigma dx.
\end{align*}
Now, notice that
\begin{align}
&\int_{\mathbb{S}^{N-1}}\frac{1}{\left\Vert \sigma\right\Vert _{K}^{N+mp}%
}\left\vert D^{m}f(x)(\sigma,...,\sigma)\right\vert ^{p}d\sigma \nonumber
=(N+mp)\int_{\mathbb{S}^{N-1}}%
{\displaystyle\int\limits_{0}^{\frac{1}{\left\Vert \sigma\right\Vert _{K}}}}
\left\vert D^{m}f(x)(\sigma,...,\sigma)\right\vert ^{p}r^{N+mp-1}%
drd\sigma\nonumber\\
&  =(N+mp)\int_{\mathbb{S}^{N-1}}%
{\displaystyle\int\limits_{0}^{\frac{1}{\left\Vert \sigma\right\Vert _{K}}}}
\left\vert D^{m}f(x)(r\sigma,...,r\sigma)\right\vert ^{p}r^{N-1}%
drd\sigma\nonumber\\
&  =(N+mp)\int_{K}\left\vert D^{m}f(x)(y,...,y)\right\vert ^{p}dy\label{id}
\end{align}
where in the last equality we used the fact that $K=\{y\in \mathbb{R}^{N}\ |\ \|y\|_K\leq 1\}$ and the conclusion follows.
\end{proof}

\begin{remark}
We explicitly note that Theorem \ref{conb} generalizes some already known
results: for example taking $m=1$ and $K=\{x\in\mathbb{R}^{n}\ |\ |x|\leq1\}$
we get \cite[Lemma 3]{nguyen06} and taking $m=2$ and $K=\{x\in\mathbb{R}%
^{n}\ |\ |x|\leq1\}$ we get \cite[Theorem 1.1]{Cui2}. Moreover, Theorem 
\ref{conb} generalizes \cite[Theorem 1.1]{NguSq} in the case where the magnetic field $A$ is zero and
$m\geq1$.
\end{remark}

\subsection{BBM formula}

In this Section we prove Theorem \ref{Bani}.\\
Let $\rho$ be a positive real function satisfying \eqref{propq}
The following result is proved in \cite[Lemma 8]{B}

\begin{lemma}\label{Sts}
Let $f\in W^{m,p}(\mathbb{R}^{N})$ with $m\geq2$ and $1\leq p<\infty$ and let
$\rho\in L^{1}(\mathbb{R})$. Then
\begin{equation*}
\int_{\mathbb{R}^{N}}\int_{\mathbb{R}^{N}} \left|  \Delta^{m}_{h} f(x)\right|
^{p} |h|^{-mp} \rho(|h|)\ dxdh\leq\frac{\|\rho\|_{L^{1}(\mathbb{R})}%
}{|\mathbb{S}^{N-1}|}\int_{\mathbb{R}^{N}}\left(  \int_{\mathbb{S}^{N-1}}
|D^{m}f(x)(\sigma,\cdots, \sigma)|^{p} d\sigma\right)  dx.
\end{equation*}

\end{lemma}

The following result is the analogous of \cite[Lemma 9]{B} in our setting.

\begin{lemma}\label{Bani1}
Fix $m\in \mathbb{N}$ and $1<p<\infty$. If $f \in C^{m+1}_{c}(\mathbb{R}^{N})$ then
\begin{equation}
\lim_{\varepsilon\to0} \int_{\mathbb{R}^{N}}\int_{\mathbb{R}^{N}} \frac{|R^{m}
f(x,y)|^{p}}{\|x-y\|^{mp}_{K}} \rho_{\varepsilon}(\|x-y\|_{K})\ dxdy =
\frac{(N+mp)}{m^{mp}}\int_{\mathbb{R}^{N}}\int_{K}|D^{m} f(x)(y,\cdots
,y)|^{p} dy dx.
\end{equation}

\end{lemma}

\begin{proof}
Let $S=\|f\|_{W^{m+1,\infty}}$. Since $t\to|t|^{p}$ is uniformly continuous in
$[0, (m+1)S]$ then for any $\delta>0$ there exists $C=C(\delta)>0$ such that
\begin{equation}
\label{daw}||s|^{p}-|t|^{p}|\leq C|s-t|+\delta\qquad\forall s,t\in[0, (m+1)S].
\end{equation}
Using \eqref{daw}, \eqref{*} and proceeding as in \cite[Lemma 9]{B} we get
\begin{align}
\label{daw2} &  \left|  \left|  \Delta^{m}_{h} f(x)\right|  ^{p}
\|h\|^{-mp}_{K}-\left|  D^{m} f(x)\left(  \frac{h}{\|h\|_{K}},\ldots, \frac
{h}{\|h\|_{K}}\right)  \right|  ^{p}\right| \\
&  \leq C \|h\|^{-m}_{K} \left|  \Delta^{m}_{h} f(x)-D^{m} f(x)(h,\ldots,
h)\right|  +\delta\nonumber\\
&  \leq C A^{-m-1} S\|h\|_{K}+\delta,\nonumber
\end{align}
for every $x\in\mathbb{R}^{N}$ and for all $h\in\mathbb{R}^{N}\setminus\{0\}$. In particular, when $h=\frac{y-x}{m}$, we get%
\begin{align}\label{picchio}
&  \left\vert m^{mp}\left\vert R^{m}f\left(  x,y\right)  \right\vert
^{p}\left\Vert y-x\right\Vert _{K}^{-mp}-\left\vert D^{m}f\left(  x\right)
\left(  \frac{y-x}{\left\Vert y-x\right\Vert _{K}},...,\frac{y-x}{\left\Vert
y-x\right\Vert _{K}}\right)  \right\vert ^{p}\right\vert \\
\nonumber
&  \leq C A^{-m-1} S\left\Vert y-x\right\Vert _{K}+\delta.
\end{align}
Let $
\mathcal{A}\left(  x,y\right)  =\left\{  \left(  x,y\right)  :\text{at least one of
}\frac{m-j}{m}x+\frac{j}{m}y, j=0,...,m,\text{ is in supp}\left(  f\right)  \right\}
$. Then we note that
\[
\int_{\mathbb{R}^{N}}\int_{\mathbb{R}^{N}}
\frac{\left\vert R^{m}f\left(  x,y\right)  \right\vert ^{p}}{\left\Vert
y-x\right\Vert _{K}^{mp}}\rho_{\varepsilon}\left(  \left\Vert y-x\right\Vert
_{K}\right)  dxdy=\underset{\mathcal{A}\left(  x,y\right)}{\int_{\mathbb{R}^{N}}\int_{\mathbb{R}^{N}}}
\frac{\left\vert R^{m}f\left(  x,y\right)  \right\vert ^{p}}{\left\Vert
y-x\right\Vert _{K}^{mp}}\rho_{\varepsilon}\left(  \left\Vert y-x\right\Vert
_{K}\right)  dxdy.
\]
Using \eqref{picchio}, we get%
\begin{align*}
&
{\displaystyle\iint\limits_{A\left(  x,y\right)  \cap\left\{  \left\vert
y-x\right\vert \leq1\right\}  }}
\frac{m^{mp}\left\vert R^{m}f\left(  x,y\right)  \right\vert ^{p}}{\left\Vert
y-x\right\Vert _{K}^{mp}}\rho_{\varepsilon}\left(  \left\Vert y-x\right\Vert
_{K}\right)  dxdy\\
&  \leq\int_{\mathbb{R}^{N}}\int_{\mathbb{R}^{N}}
\left\vert D^{m}f\left(  x\right)  \left(  \frac{y-x}{\left\Vert
y-x\right\Vert _{K}},...,\frac{y-x}{\left\Vert y-x\right\Vert _{K}}\right)
\right\vert ^{p}\rho_{\varepsilon}\left(  \left\Vert y-x\right\Vert
_{K}\right)  dxdy\\
&  +C A^{-m-1} S%
{\displaystyle\iint\limits_{\mathcal{A}\left(  x,y\right)  \cap\left\{  \left\vert
y-x\right\vert \leq1\right\}  }}
\left\Vert y-x\right\Vert _{K}\rho_{\varepsilon}\left(  \left\Vert
y-x\right\Vert _{K}\right)  +\delta%
{\displaystyle\iint\limits_{\mathcal{A}\left(  x,y\right)  \cap\left\{  \left\vert
y-x\right\vert \leq1\right\}  }}
\rho_{\varepsilon}\left(  \left\Vert y-x\right\Vert _{K}\right)  dxdy
\end{align*}
Note that
\begin{align*}%
{\displaystyle\iint\limits_{\mathcal{A}\left(  x,y\right)  \cap\left\{  \left\vert
y-x\right\vert \leq1\right\}  }}
\left\Vert y-x\right\Vert _{K}\rho_{\varepsilon}\left(  \left\Vert
y-x\right\Vert _{K}\right)   &  =%
{\displaystyle\int\limits_{\left\vert h\right\vert \leq1}}
\left\Vert h\right\Vert _{K}\rho_{\varepsilon}\left(  \left\Vert h\right\Vert
_{K}\right)
{\displaystyle\int\limits_{\mathcal{A}\left(  x,x+h\right)  }}
dxdh\\
&  \leq%
{\displaystyle\int\limits_{\left\vert h\right\vert \leq1}}
\left\Vert h\right\Vert _{K}\rho_{\varepsilon}\left(  \left\Vert h\right\Vert
_{K}\right)  \left(  m+1\right)  \left\vert \text{supp}\left(  f\right)
\right\vert dh\\
&=(m+1)\text{supp}\left(  f\right) \int\limits_{\mathbb{S}^{N-1}}\|\sigma\|_K^{-N}d\sigma\int\limits_{0}^1\rho_{\varepsilon}(s) s^N ds.
\end{align*}
By \eqref{propq} we automatically get that \cite[Remark 4.3]{PinSquVec} $
\lim_{\varepsilon\to 0}\int\limits_{0}^1\rho_{\varepsilon}(s) s^N ds=0$, 
thus
\[
\lim_{\varepsilon\to 0}{\displaystyle\iint\limits_{\mathcal{A}\left(  x,y\right)  \cap\left\{  \left\vert
y-x\right\vert \leq1\right\}  }}
\left\Vert y-x\right\Vert _{K}\rho_{\varepsilon}\left(  \left\Vert
y-x\right\Vert _{K}\right)=0.
\]
Similarly,
\[
\delta%
{\displaystyle\iint\limits_{\mathcal{A}\left(  x,y\right)  \cap\left\{  \left\vert
y-x\right\vert \leq1\right\}  }}
\rho_{\varepsilon}\left(  \left\Vert y-x\right\Vert _{K}\right)
dxdy\lesssim\delta
\]
On the other hand,
\begin{align*}
&
{\displaystyle\iint\limits_{\mathcal{A}\left(  x,y\right)  \cap\left\{  \left\vert
y-x\right\vert >1\right\}  }}
\frac{m^{mp}\left\vert R^{m}f\left(  x,y\right)  \right\vert ^{p}}{\left\Vert
y-x\right\Vert _{K}^{mp}}\rho_{\varepsilon}\left(  \left\Vert y-x\right\Vert
_{K}\right)  dxdy\\
&  \lesssim\left\Vert f\right\Vert _{p}%
{\displaystyle\int\limits_{\left\vert h\right\vert >1}}
\rho_{\varepsilon}\left(  \left\Vert h\right\Vert \right)  dh\rightarrow0 \text{ as } \varepsilon \rightarrow0.
\end{align*}
Hence, by sending $\varepsilon \rightarrow0$ and then $\delta \rightarrow0$, we can now conclude that
\begin{align*}
&  \limsup_{\varepsilon\to0}
\int_{\mathbb{R}^{N}}\int_{\mathbb{R}^{N}}
\frac{\left\vert R^{m}f\left(  x,y\right)  \right\vert ^{p}}{\left\Vert
y-x\right\Vert _{K}^{mp}}\rho_{\varepsilon}\left(  \left\Vert y-x\right\Vert
_{K}\right)  dxdy\\
&  \leq\frac{1}{m^{mp}}\limsup_{\varepsilon\to0}
\int_{\mathbb{R}^{N}}\int_{\mathbb{R}^{N}}
\left\vert D^{m}f\left(  x\right)  \left(  \frac{y-x}{\left\Vert
y-x\right\Vert _{K}},...,\frac{y-x}{\left\Vert y-x\right\Vert _{K}}\right)
\right\vert ^{p}\rho_{\varepsilon}\left(  \left\Vert y-x\right\Vert
_{K}\right)  dxdy.
\end{align*}
We next compute the limit
of the quantity on the right-hand side. We have
\begin{align}\label{kill}
&  \int_{\mathbb{R}^{N}}\left|  D^{m} f(x)\left(  \frac{x-y}{\|x-y\|_K},\ldots,
\frac{x-y}{\|x-y\|_K}\right)  \right|  ^{p} \rho_{\varepsilon}%
(\|x-y\|_K)\ dy\\
&  =\int_{\mathbb{S}^{N-1}}\int_{0}^{\infty}\left|  D^{m} f(x)\left(
\sigma,\ldots, \sigma\right)  \right|  ^{p}\frac{1}{\Vert\sigma\Vert_{K}^{mp}}
\rho_{\varepsilon}(r\Vert\sigma\Vert_{K})r^{N-1}\ drd\sigma\nonumber\\
&  =\int_{\mathbb{S}^{N-1}}\frac{1}{\Vert\sigma\Vert_{K}^{N+mp}}\left|  D^{m}
f(x)\left(  \sigma,\ldots, \sigma\right)  \right|  ^{p}\ d\sigma\nonumber\\
&  =(N+mp)\int_{K}\left\vert D^{m}f(x)(y,...,y)\right\vert
^{p}dy,\nonumber
\end{align}
where the last equality follows from \eqref{id}. Thus, 
$$\limsup_{\varepsilon\to0}
\int_{\mathbb{R}^{N}}\int_{\mathbb{R}^{N}}
\frac{\left\vert R^{m}f\left(  x,y\right)  \right\vert ^{p}}{\left\Vert
y-x\right\Vert _{K}^{mp}}\rho_{\varepsilon}\left(  \left\Vert y-x\right\Vert
_{K}\right)  dxdy\\
\leq\frac{N+mp}{m^{mp}}\int_{\mathbb{R}^{N}}\int_{K}|D^{m} f(x)(y,\cdots,y)|^{p}
dy dx.
$$
Let $R>0$ be such that  supp$(f)\subset B_{R}$, then%
\begin{align*}
&  \underset{\left\vert y-x\right\vert \leq1}{%
{\displaystyle\iint}
}\left\vert D^{m}f\left(  x\right)  \left(  \frac{y-x}{\left\Vert
y-x\right\Vert _{K}},...,\frac{y-x}{\left\Vert y-x\right\Vert _{K}}\right)
\right\vert ^{p}\rho_{\varepsilon}\left(  \left\Vert y-x\right\Vert
_{K}\right)  dxdy\\
&  =\underset{\left\vert y-x\right\vert \leq1}{%
{\displaystyle\int\limits_{B_{R}}}
{\displaystyle\int\limits_{\mathbb{R}^{N}}}
}\left\vert D^{m}f\left(  x\right)  \left(  \frac{y-x}{\left\Vert
y-x\right\Vert _{K}},...,\frac{y-x}{\left\Vert y-x\right\Vert _{K}}\right)
\right\vert ^{p}\rho_{\varepsilon}\left(  \left\Vert y-x\right\Vert
_{K}\right)  dydx\\
&  \leq\int_{\mathbb{R}^{N}}\int_{\mathbb{R}^{N}}
\frac{m^{mp}\left\vert R^{m}f\left(  x,y\right)  \right\vert ^{p}}{\left\Vert
y-x\right\Vert _{K}^{mp}}\rho_{\varepsilon}\left(  \left\Vert y-x\right\Vert
_{K}\right)  dxdy\\
&  +\underset{\left\vert y-x\right\vert \leq1}{%
{\displaystyle\int\limits_{B_{R}}}
{\displaystyle\int\limits_{\mathbb{R}^{N}}}
}C_{\delta}S\left\Vert y-x\right\Vert _{K}\rho_{\varepsilon}\left(  \left\Vert
y-x\right\Vert _{K}\right)  dydx+\underset{\left\vert y-x\right\vert \leq1}{%
{\displaystyle\int\limits_{B_{R}}}
{\displaystyle\int\limits_{\mathbb{R}^{N}}}
}\rho_{\varepsilon}\left(  \left\Vert y-x\right\Vert _{K}\right)  dxdy
\end{align*}
As above
\begin{align*}
\underset{\left\vert y-x\right\vert \leq1}{%
{\displaystyle\int\limits_{B_{R}}}
{\displaystyle\int\limits_{\mathbb{R}^{N}}}
}\left\Vert y-x\right\Vert _{K}\rho_{\varepsilon}\left(  \left\Vert
y-x\right\Vert _{K}\right)  dydx &  \rightarrow0 \text{ as } \varepsilon \rightarrow0,\\
\underset{\left\vert y-x\right\vert \leq1}{%
{\displaystyle\int\limits_{B_{R}}}
{\displaystyle\int\limits_{\mathbb{R}^{N}}}
}\rho_{\varepsilon}\left(  \left\Vert y-x\right\Vert _{K}\right)  dxdy &
\lesssim\delta.
\end{align*}
By \eqref{kill}
\begin{align*}
  &\underset{\left\vert y-x\right\vert \leq1}{{\displaystyle\int\limits_{B_{R}}}
{\displaystyle\int\limits_{\mathbb{R}^{N}}}
}\left\vert D^{m}f\left(  x\right)  \left(  \frac{y-x}{\left\Vert
y-x\right\Vert _{K}},\cdots,\frac{y-x}{\left\Vert y-x\right\Vert _{K}}\right)
\right\vert ^{p}\rho_{\varepsilon}\left(  \left\Vert y-x\right\Vert_{K}\right)  dydx\\
&=(N+mp)\int_{\mathbb{R}^N}\int_K|D^mf(x)(y,\cdots,y)|^p dydx
\end{align*}
Thus,
$$
\frac{N+mp}{m^{mp}}\int_{\mathbb{R}^{N}}\int_{K}|D^{m} f(x)(y,\cdots,y)|^{p}
dy dx\\
 \leq \limsup_{\varepsilon\to0}
\int_{\mathbb{R}^{N}}\int_{\mathbb{R}^{N}}
\frac{\left\vert R^{m}f\left(  x,y\right)  \right\vert ^{p}}{\left\Vert
y-x\right\Vert _{K}^{mp}}\rho_{\varepsilon}\left(  \left\Vert y-x\right\Vert
_{K}\right)  dxdy,
$$
We conclude that
$$ \lim_{\varepsilon\to0}
\int_{\mathbb{R}^{N}}\int_{\mathbb{R}^{N}}
\frac{\left\vert R^{m}f\left(  x,y\right)  \right\vert ^{p}}{\left\Vert
y-x\right\Vert _{K}^{mp}}\rho_{\varepsilon}\left(  \left\Vert y-x\right\Vert
_{K}\right)  dxdy\\
=\frac{N+mp}{m^{mp}}\int_{\mathbb{R}^{N}}\int_{K}|D^{m} f(x)(y,\cdots
,y)|^{p} dy dx.
$$
\end{proof}

\begin{proof}[\textbf{Proof of Theorem \ref{Bani}}]
 So now we consider $f\in W^{m,p}(\mathbb{R}^{N})$ and let $f_{n}\in C_{c}^{\infty}(\mathbb{R}^{N})$ such that $f_{n}\to f$ in the $W^{m,p}(\mathbb{R}^{N})$ norm. Then one has
\begin{align}
&\Big|\Big(\int_{\mathbb{R}^{N}}\int_{\mathbb{R}^{N}}\left|  R^{m} f(x,y)\right|  ^{p}
\|x-y\|^{-mp}_{K}\rho_{\varepsilon}(\|x-y\|_{K})\ dydx\Big)^{\frac{1}{p}}\notag\\
&-\Big(\int_{\mathbb{R}^{N}}\int_{\mathbb{R}^{N}}\left|  R^{m} f_{n}(x,y)\right|  ^{p}
\|x-y\|^{-mp}_{K}\rho_{\varepsilon}(\|x-y\|_{K})\ dydx\Big)^{\frac{1}{p}}\Big|\notag\\
&\leq \Big(\int_{\mathbb{R}^{N}}\int_{\mathbb{R}^{N}}\left|  R^{m} (f-f_{n})(x,y)\right|  ^{p}
\|x-y\|^{-mp}_{K}\rho_{\varepsilon}(\|x-y\|_{K})\ dydx\Big)^{\frac{1}{p}}\notag\\
&\lesssim \Big(\int_{\mathbb{R}^{N}}\left(  \int_{\mathbb{S}^{N-1}}
|D^{m}(f_{n}-f)(x)(\sigma,\cdots, \sigma)|^{p} d\sigma\right) dx\Big)^{\frac{1}{p}}.\notag
\end{align}
Where we used Lemma \ref{Sts} in the last inequality. Thus 
\begin{align}
&\Big(\int_{\mathbb{R}^{N}}\int_{\mathbb{R}^{N}}\left|  R^{m} f(x,y)\right|  ^{p}
\|x-y\|^{-mp}_{K}\rho_{\varepsilon}(\|x-y\|_{K})\ dydx\Big)^{\frac{1}{p}}=\notag\\
&\Big(\int_{\mathbb{R}^{N}}\int_{\mathbb{R}^{N}}\left|  R^{m} f_{n}(x,y)\right|  ^{p}
\|x-y\|^{-mp}_{K}\rho_{\varepsilon}(\|x-y\|_{K})\ dydx\Big)^{\frac{1}{p}}+o(1)\notag
\end{align}
here $o(1)\to 0$ as $n\to \infty$ uniformly on $\varepsilon$. So fix $\epsilon>0$, then there exists $n_{0}$ big enough so that for $n\geq n_{0}$, we have $\|f-f_{n}\|_{W^{m,p}}<\epsilon$ and
\begin{align}
&\Big|\Big(\int_{\mathbb{R}^{N}}\int_{\mathbb{R}^{N}}\left|  R^{m} f(x,y)\right|  ^{p}
\|x-y\|^{-mp}_{K}\rho_{\varepsilon}(\|x-y\|_{K})\ dydx\Big)^{\frac{1}{p}}\notag\\
&-\Big(\int_{\mathbb{R}^{N}}\int_{\mathbb{R}^{N}}\left|  R^{m} f_{n}(x,y)\right|  ^{p}
\|x-y\|^{-mp}_{K}\rho_{\varepsilon}(\|x-y\|_{K})\ dydx\Big)^{\frac{1}{p}}\Big|<\epsilon.\notag
\end{align}
Then we have 
\begin{align}
&\lim_{\varepsilon\to 0}\Big|\Big(\int_{\mathbb{R}^{N}}\int_{\mathbb{R}^{N}}\left|  R^{m} f(x,y)\right|  ^{p}
\|x-y\|^{-mp}_{K}\rho_{\varepsilon}(\|x-y\|_{K})\ dydx\Big)^{\frac{1}{p}}\notag\\
&-\Big(\frac{N+mp}{m^{mp}}\int_{\mathbb{R}^{N}}\int_{K}|D^{m} f(x)(y,\cdots,y)|^{p}
dy dx)^{\frac{1}{p}}\Big|\notag\\
&\leq \lim_{\varepsilon\to 0}\Big|\Big(\int_{\mathbb{R}^{N}}\int_{\mathbb{R}^{N}}\left|  R^{m} f(x,y)\right|  ^{p}
\|x-y\|^{-mp}_{K}\rho_{\varepsilon}(\|x-y\|_{K})\ dydx\Big)^{\frac{1}{p}}\notag\\
&-\Big(\int_{\mathbb{R}^{N}}\int_{\mathbb{R}^{N}}\left|  R^{m} f_{n}(x,y)\right|  ^{p}
\|x-y\|^{-mp}_{K}\rho_{\varepsilon}(\|x-y\|_{K})\ dydx\Big)^{\frac{1}{p}}\Big|\notag\\
&+\lim_{\varepsilon\to 0}\Big|\Big(\int_{\mathbb{R}^{N}}\int_{\mathbb{R}^{N}}\left|  R^{m} f_{n}(x,y)\right|  ^{p}
\|x-y\|^{-mp}_{K}\rho_{\varepsilon}(\|x-y\|_{K})\ dydx\Big)^{\frac{1}{p}}\notag\\
&-\Big(\frac{N+mp}{m^{mp}}\int_{\mathbb{R}^{N}}\int_{K}|D^{m} f_{n}(x)(y,\cdots,y)|^{p}
dy dx\Big)^{\frac{1}{p}}\Big|\notag\\
&+\Big|\Big(\frac{N+mp}{m^{mp}}\int_{\mathbb{R}^{N}}\int_{K}|D^{m} f_{n}(x)(y,\cdots,y)|^{p}
dy dx\Big)^{\frac{1}{p}}\\
&-\Big(\frac{N+mp}{m^{mp}}\int_{\mathbb{R}^{N}}\int_{K}|D^{m} f(x)(y,\cdots,y)|^{p}
dy dx\Big)^{\frac{1}{p}}\Big|\notag\\
&\lesssim  \lim_{\varepsilon\to 0}\Big|\Big(\int_{\mathbb{R}^{N}}\int_{\mathbb{R}^{N}}\left|  R^{m} f_{n}(x,y)\right|  ^{p}
\|x-y\|^{-mp}_{K}\rho_{\varepsilon}(\|x-y\|_{K})\ dydx\Big)^{\frac{1}{p}}\notag\\
&-\Big(\frac{N+mp}{m^{mp}}\int_{\mathbb{R}^{N}}\int_{K}|D^{m} f_{n}(x)(y,\cdots,y)|^{p}
dy dx\Big)^{\frac{1}{p}}\Big|+2\epsilon \notag
\end{align}
So the conclusion follows from Theorem \ref{Bani1}.

\end{proof}

\section{Characterizations of the higher order Sobolev spaces via the Taylor
remainder}

We recall that
\[
T_{y}^{m-1}f\left(  x\right)  =%
{\displaystyle\sum\limits_{\left\vert \alpha\right\vert \leq m-1}}
D^{\alpha}f\left(  y\right)  \frac{\left(  x-y\right)  ^{\alpha}}{\alpha!}%
\]
and
\[
R_{m-1}f\left(  x,y\right)  =f\left(  x\right)  -T_{y}^{m-1}f\left(  x\right)  .
\]
Proceeding as in \cite{Cui2} and by an easy induction we get
\begin{equation}\label{stimaT}
R_{m-1} f(x,x+he_N)= h^m \int_{[0,1]^m}\partial^m_{x_N} f(x', x_N+\prod_{i=1}^m t_i h) \prod_{i=1}^m t_i^{m-i}\, dt_1\ldots dt_m
\end{equation}
thus
\begin{equation}\label{killall}
|R_{m-1} f(x,x+he_N)|\leq \frac{h^m}{m!}\mathbb{M}_N(\partial^m_{x_N}f)(x).
\end{equation}
\subsection{Proof of Theorem \ref{lam1}}
Using \eqref{killall} the proof of the following Lemma is very similar to the one of Lemma \ref{pols} so we omit it. 
\begin{lemma}Let $K\subset\mathbb{R}^N$ be a convex, symmetric set containing the origin with nonempty interior.
Let $f\in W^{m,p}(\mathbb{R}^{N})$, $1<p<\infty$ and $m\in\mathbb{N}$. Then there exists a constant
$C=C(m,N,p)>0$ such that

\[
\underset{\left\vert R_{m-1}f\left(  x,y\right)  \right\vert >\delta}{%
{\displaystyle\int\limits_{\mathbb{R} ^{N}}}
{\displaystyle\int\limits_{\mathbb{R} ^{N}}}
}\frac{\delta^{p}}{\left\Vert x-y\right\Vert _{K}^{N+mp}}dxdy\leq C%
{\displaystyle\int\limits_{\mathbb{R} ^{N}}}
|\nabla^{m}f\left(  x\right)  |^{p}dx,\forall\delta>0.
\]

\end{lemma}

We are now ready to prove Theorem \ref{lam1}.\\




\begin{proof}[\textbf{Proof of Theorem \ref{lam1}:}] Using a change of variables (writing $y=x+\sqrt[m]{\delta}h\sigma$,
$\sigma=\frac{y-x}{\left\vert y-x\right\vert }$, $h=\frac{1}{\sqrt[m]{\delta}%
}\left\vert y-x\right\vert $), we obtain
\[
\underset{\left\vert R_{m-1}f\left(  x,y\right)  \right\vert >\delta}{%
{\displaystyle\int\limits_{\mathbb{R}^{N}}}
{\displaystyle\int\limits_{\mathbb{R}^{N}}}
}\frac{\delta^{p}}{\left\Vert x-y\right\Vert _{K}^{N+mp}}dxdy=\underset
{|\frac{R_{m-1}f\left(  x,x+\sqrt[m]{\delta}h\sigma\right)  }{h^{m}\delta
}|h^{m}>1}{\int_{\mathbb{S}^{N-1}}\int_{\mathbb{R}^{N}}\int_{0}^{\infty}}%
\frac{1}{\left\Vert \sigma\right\Vert _{K}^{N+mp}}\frac{1}{h^{mp+1}%
}dhdxd\sigma.
\]
We define the auxiliary function $F_{\delta}:\mathbb{S}^{N-1}\rightarrow\mathbb{R}$ by
\[
F_{\delta}(\sigma):=\frac{1}{\left\Vert \sigma\right\Vert _{K}^{N+mp}}%
\underset{|\frac{R_{m-1}f\left(  x,x+\sqrt[m]{\delta}h\sigma\right)  }%
{h^{m}\delta}|h^{m}>1}{\int_{\mathbb{R}^{N}}\int_{0}^{\infty}}\frac
{1}{h^{mp+1}}dhdx.
\]
We first prove that for all $\sigma\in\mathbb{S}^{N-1}$, $\forall\delta>0$
\begin{equation}
F_{\delta}(\sigma)\leq\frac{1}{A^{N+mp}}C(m,N,p)%
{\displaystyle\int\limits_{\mathbb{R}^{N}}}
|\nabla^{m}f\left(  x\right)  |^{p}dx. \label{3.3}%
\end{equation}
that is
\[
\underset{\left\vert \frac{R_{m-1}f\left(  x,x+\sqrt[m]{\delta}h\sigma\right)
}{h^{m}\delta}\right\vert h^{m}>1}{\int_{\mathbb{R}^{N}}\int_{0}^{\infty}%
}\frac{1}{h^{mp+1}}dhdx\leq C(m,N,p)%
{\displaystyle\int\limits_{\mathbb{R}^{N}}}
|\nabla^{m}f\left(  x\right)  |^{p}dx.
\]
Without loss of generality, we assume that $\sigma=e_{N}=(0,...,0,1)$.
Hence, we need to verify that
\begin{equation}
\underset{\left\vert \frac{R_{m-1}f\left(  x,x+\sqrt[m]{\delta}he_{N}\right)
}{h^{m}\delta}\right\vert h^{m}>1}{\int_{\mathbb{R}^{N}}\int_{0}^{\infty}%
}\frac{1}{h^{mp+1}}dhdx\leq C(m,N,p)%
{\displaystyle\int\limits_{\mathbb{R}^{N}}}
|\nabla^{m}f\left(  x\right)  |^{p}dx. \label{3.41}%
\end{equation}
By \eqref{killall} one has
\[
\left\vert \frac{R_{m-1}f\left(  x,x+\sqrt[m]{\delta}he_{N}\right)  }%
{h^{m}\delta}\right\vert \leq\frac{1}{m!}\mathbb{M}_{N}\left(  \partial_{x_{N}}^{m}f\right)
 \left(
x\right)
\]
and therefore
\begin{align*}
\underset{\left\vert \frac{R_{m-1}f\left(  x,x+\sqrt[m]{\delta}he_{N}\right)
}{h^{m}\delta}\right\vert h^{m}>1}{\int_{\mathbb{R}^{N}}\int_{0}^{\infty}%
}\frac{1}{h^{mp+1}}dhdx  &  \leq\underset{h^{m}\mathbb{M}_{N}\left(  \partial_{x_{N}}^{m}f\right)
 \left(x\right)>m!}{\int_{\mathbb{R}^{N}}\int_{0}^{\infty}}\frac
{1}{h^{mp+1}}dhdx\\
&  \leq\frac{\left(  m!\right)  ^{p}}{mp}%
{\displaystyle\int\limits_{\mathbb{R}^{N}}}
\left\vert \mathbb{M}_{N}\left(  \partial_{x_{N}}^{m}f\right) \left(x\right)\right\vert
^{p}dx\\
&  \leq C(m,N,p)%
{\displaystyle\int\limits_{\mathbb{R}^{N}}}
|\nabla^{m}f\left(  x\right)  |^{p}dx.
\end{align*}
Next we show that
\begin{equation}
F_{\delta}(\sigma)\rightarrow\frac{1}{\left(  m!\right)  ^{p}mp}\frac
{1}{\left\Vert \sigma\right\Vert _{K}^{N+mp}}%
{\displaystyle\int\limits_{\mathbb{R}^{N}}}
|D^{m}f(x)(\sigma,...,\sigma)|^{p}dx\text{ as }\delta\rightarrow0\text{ for
every }\sigma\in\mathbb{S}^{N-1}. \label{3.5}%
\end{equation}
It is enough to show
\[
\int_{\mathbb{R}^{N}}\int_{0}^{\infty}G_{\delta}(x,h)dhdx\rightarrow\frac
{1}{\left(  m!\right)  ^{p}mp}%
{\displaystyle\int\limits_{\mathbb{R}^{N}}}
|D^{m}f(x)(\sigma,...,\sigma)|^{p}dx\text{ as }\delta\rightarrow0,
\]
where
\[
G_{\delta}(x,h):=\frac{1}{h^{mp+1}}1_{\left\{  |\frac{R_{m-1}f\left(
x,x+\sqrt[m]{\delta}h\sigma\right)  }{h^{m}\delta}|h^{m}>1\right\}  }(x,h).
\]
With loss of generality, we suppose that $\sigma=e_{N}=(0,...,0,1)$. Noting
that for all $\sigma\in\mathbb{S}^{N-1}$: $G_{\delta}(x,h)\rightarrow\frac
{1}{h^{2p+1}}1_{\{|D^{m}f(x)(\sigma,...,\sigma)|h^{m}>m!\}}$ as
$\delta\rightarrow0$ for a.e. $(x,h)\in\mathbb{R}^{N}\times\lbrack0,\infty)$,
and
\[
G_{\delta}(x,h)\leq\frac{1}{h^{mp+1}}1_{\{h^{m}\mathbb{M}_{N}\left(  \partial_{x_{N}}^{m}f\right)\left(x\right)>m!\}}(x,h)\in\mathit{L}^{1}(\mathbb{R}^{N}%
\times\lbrack0,\infty)).
\]
Hence, by the Lebesgue dominated convergence theorem, we get (\ref{3.5}). 
Once again, by the Lebesgue dominated convergence theorem, we conclude that
\begin{align*}
&\underset{\delta\rightarrow0}{\lim}\underset{\left\vert R_{m-1}f\left(
x,y\right)  \right\vert >\delta}{\int_{\mathbb{R}^{N}}\int_{\mathbb{R}^{N}}%
}\frac{\delta^{p}}{\left\Vert x-y\right\Vert _{K}^{N+mp}}dxdy\\  &
=\int_{\mathbb{S}^{N-1}}\frac{1}{\left(  m!\right)  ^{p}mp}\frac{1}{\left\Vert
\sigma\right\Vert _{K}^{N+mp}}\int_{\mathbb{R}^{N}}\left\vert D^{m}%
f(x)(\sigma,...,\sigma)\right\vert ^{p}dxd\sigma\\
&  =\frac{1}{\left(  m!\right)  ^{p}mp}\int_{\mathbb{R}^{N}}\int
_{\mathbb{S}^{N-1}}\frac{1}{\left\Vert \sigma\right\Vert _{K}^{N+mp}}\left\vert
D^{m}f(x)(\sigma,...,\sigma)\right\vert ^{p}d\sigma dx\\
&=\frac{N+mp}{\left(
m!\right)  ^{p}mp}\int_{\mathbb{R}^{N}}\int_{K}\left\vert
D^{m}f(x)(y,...,y)\right\vert ^{p}dydx
\end{align*}

\end{proof}

\subsection{Proof of Theorem \ref{lam2}}
 The mollifiers $\rho_{\varepsilon}\in L_{loc}^{1}\left(  0,\infty\right) $ are nonnegative functions satisfying, 

$${\displaystyle\int\limits_{0}^{\infty}}
\rho_{\varepsilon}\left(  r\right)  r^{N-1}dr=1 \quad \text{and}\quad \lim_{\varepsilon
\rightarrow0}
{\displaystyle\int\limits_{\delta}^{\infty}}
\rho_{\varepsilon}\left(  r\right)  r^{N-1}dr=0 \text{ for all } \delta>0.$$
\begin{lemma}
Let $K\subset\mathbb{R}^N$ be a convex, symmetric set containing the origin with nonempty interior.
Let $f\in W^{m,p}(\mathbb{R}^{N})$, $1<p<\infty$ and $m\in\mathbb{N}$. Then there exists a constant
$C=C(m,N,p)>0$ such that for all $\varepsilon>0$ the following inequality holds
\[
\int_{\mathbb{R}^{N}}\int_{\mathbb{R}^{N}}\frac{\left\vert R_{m-1}f\left(
x,y\right)  \right\vert ^{p}}{\left\Vert x-y\right\Vert _{K}^{mp}}%
\rho_{\varepsilon}\left(  \left\Vert x-y\right\Vert _{K}\right)  dxdy\leq
C\int_{\mathbb{R}^{N}}|\nabla^{m}f\left(  x\right)  |^{p}dx.
\]

\end{lemma}
\begin{proof}
By density we can assume that $f\in C_{c}^{\infty}\left(
\mathbb{R}
^{N}\right)  $. We have%
\begin{align*}
\int_{\mathbb{R}^{N}}\int_{\mathbb{R}^{N}}\frac{\left\vert R_{m-1}f\left(
x,y\right)  \right\vert ^{p}}{\left\Vert x-y\right\Vert _{K}^{mp}}%
\rho_{\varepsilon}\left(  \left\Vert x-y\right\Vert _{K}\right)  dxdy  &
=\int_{\mathbb{R}^{N}}\int_{\mathbb{R}^{N}}\frac{\left\vert R_{m-1}f\left(
x+h,x\right)  \right\vert ^{p}}{\left\Vert h\right\Vert _{K}^{mp}}%
\rho_{\varepsilon}\left(  \left\Vert h\right\Vert _{K}\right)  dxdh\\
&  =\int_{\mathbb{R}^{N}}\frac{\rho_{\varepsilon}\left(  \left\Vert
h\right\Vert _{K}\right)  }{\left\Vert h\right\Vert _{K}^{mp}}\int
_{\mathbb{R}^{N}}\left\vert R_{m-1}f\left(  x+h,x\right)  \right\vert ^{p}dxdh.
\end{align*}
Since by \cite[(2.17)]{BIK2011} there exists $C=C(m,N,p)>0$ such that
\begin{align*}
\int_{\mathbb{R}^{N}}\left\vert R_{m-1}f\left(  x+h,x\right)  \right\vert
^{p}dx\leq C\left\vert h\right\vert ^{mp}\int_{\mathbb{R}^{N}}%
|\nabla^{m}f\left(  x\right)  |^{p}dx,
\end{align*}
we have%
\[
\int_{\mathbb{R}^{N}}\int_{\mathbb{R}^{N}}\frac{\left\vert R_{m-1}f\left(
x,y\right)  \right\vert ^{p}}{\left\Vert x-y\right\Vert _{K}^{mp}}%
\rho_{\varepsilon}\left(  \left\Vert x-y\right\Vert _{K}\right)  dxdy\leq
\frac{1}{A^{mp}}C\int_{\mathbb{R}^{N}}\rho_{\varepsilon}\left(
\left\Vert h\right\Vert _{K}\right)  dh\int_{\mathbb{R}^{N}}|\nabla
^{m}f\left(  x\right)  |^{p}dx.
\]
On the other hand%
\begin{align*}
\int_{\mathbb{R}^{N}}\rho_{\varepsilon}\left(  \left\Vert h\right\Vert
_{K}\right)  dh  &  =\int_{\mathbb{S}^{N-1}}\int_{0}^{\infty}\rho_{\varepsilon
}\left(  r\left\Vert \sigma\right\Vert _{K}\right)  r^{N-1}drd\sigma\\
&  =\int_{\mathbb{S}^{N-1}}\int_{0}^{\infty}\rho_{\varepsilon}\left(  s\right)
\left(  \frac{s}{\left\Vert \sigma\right\Vert _{K}}\right)  ^{N-1}\frac
{1}{\left\Vert \sigma\right\Vert _{K}}dsd\sigma\\
&  \leq\frac{1}{A^{mp}},
\end{align*}
and the conclusion follows.
\end{proof}




\begin{proof}[\textbf{Proof of Theorem \ref{lam2}:}]

First, we assume that $f\in C_{c}^{\infty}\left(
\mathbb{R}
^{N}\right)  .$ We have
\[
\int_{\mathbb{R}^{N}}\int_{\mathbb{R}^{N}}\frac{\left\vert R_{m-1}f\left(
x,y\right)  \right\vert ^{p}}{\left\Vert x-y\right\Vert _{K}^{mp}}%
\rho_{\varepsilon}\left(  \left\Vert x-y\right\Vert _{K}\right)
dxdy=\int_{\mathbb{R}^{N}}\int_{\mathbb{R}^{N}}\frac{\left\vert R_{m-1}%
f\left(  x,x+h\right)  \right\vert ^{p}}{\left\Vert h\right\Vert _{K}^{mp}%
}\rho_{\varepsilon}\left(  \left\Vert h\right\Vert _{K}\right)  dxdh
\]
By Taylor's formula, we have that for every $\delta>0$, there exists
$C_{\delta}>0$ such that
\[
\left\vert R_{m-1}f\left(  x+h,x\right)  \right\vert ^{p}\leq\left(
1+\delta\right)  \frac{1}{\left(  m!\right)  ^{p}}\left\vert D^{m}%
f(x)(h,...,h)\right\vert ^{p}+C_{\delta}\left\vert h\right\vert ^{\left(
m+1\right)  p}.
\]
Hence%
\begin{align*}
& \underset{\left\vert h \right\vert \leq 1}{ \int_{\mathbb{R}^{N}}\int_{\mathbb{R}^{N}}}\frac{\left\vert R_{m-1}f\left(
x+h,x\right)  \right\vert ^{p}}{\left\Vert h\right\Vert _{K}^{mp}}%
\rho_{\varepsilon}\left(  \left\Vert h\right\Vert _{K}\right)  dxdh\\
&  \leq\frac{1+\delta}{\left(  m!\right)  ^{p}}
\int_{\mathbb{R}^{N}}\int_{\mathbb{R}^{N}}\left\vert D^{m}%
f(x)(h,...,h)\right\vert ^{p}\frac{\rho_{\varepsilon}\left(  \left\Vert
h\right\Vert _{K}\right)  }{\left\Vert h\right\Vert _{K}^{mp}}dxdh\\
&  +C_{\delta}\underset{\left\vert h \right\vert \leq 1}{\int_{\mathbb{R}^{N}}}\left\vert h\right\vert ^{\left(
m+1\right)  p}\frac{\rho_{\varepsilon}\left(  \left\Vert h\right\Vert
_{K}\right)  }{\left\Vert h\right\Vert _{K}^{mp}}%
{\displaystyle\int\limits_{\left\{  x\text{ or }x+h\in\text{supp}\left(
f\right)  \right\}  }}
dxdh\\
&  \leq\frac{1+\delta}{\left(  m!\right)  ^{p}} \int
_{\mathbb{R}^{N}}\int_{\mathbb{R}^{N}}\left\vert D^{m}f(x)(\frac{h}{\left\vert
h\right\vert },...,\frac{h}{\left\vert h\right\vert })\right\vert
^{p}\left\vert h\right\vert ^{mp}\frac{\rho_{\varepsilon}\left(  \left\Vert
h\right\Vert _{K}\right)  }{\left\Vert h\right\Vert _{K}^{mp}}dxdh\\
&  +2C_{\delta}\left\vert \text{supp}\left(  f\right)  \right\vert
\underset{\left\vert h \right\vert \leq 1}{\int_{\mathbb{R}^{N}}}\left\vert h\right\vert ^{\left(  m+1\right)  p}%
\frac{\rho_{\varepsilon}\left(  \left\Vert h\right\Vert _{K}\right)
}{\left\Vert h\right\Vert _{K}^{mp}}dh.
\end{align*}
By \eqref{*} and using \eqref{propq} it is easy to see that
\begin{align*}
\lim_{\varepsilon\to 0}\underset{\left\vert h \right\vert \leq 1}{\int_{\mathbb{R}^{N}}}\left\vert h\right\vert ^{\left(  m+1\right)  p}%
\frac{\rho_{\varepsilon}\left(  \left\Vert h\right\Vert _{K}\right)
}{\left\Vert h\right\Vert _{K}^{mp}}dh  =0.
\end{align*}
On the other hand,%
\begin{align*}
&  \frac{1}{\left(  m!\right)  ^{p}}\int_{\mathbb{R}^{N}}\int_{\mathbb{R}^{N}%
}\left\vert D^{m}f(x)(\frac{h}{\left\vert h\right\vert },...,\frac
{h}{\left\vert h\right\vert })\right\vert ^{p}\left\vert h\right\vert
^{mp}\frac{\rho_{\varepsilon}\left(  \left\Vert h\right\Vert _{K}\right)
}{\left\Vert h\right\Vert _{K}^{mp}}dxdh\\
&  =\frac{N+mp}{\left(  m!\right)  ^{p}}\int_{\mathbb{R}^{N}}%
\int_{K}\left\vert D^{m}f(x)(y,...,y)\right\vert ^{p}dydx.
\end{align*}
Also, since $f\in C_{c}^{\infty}\left(
\mathbb{R}
^{N}\right)$, we have
\begin{align*}
 \underset{\left\vert h \right\vert \geq 1}{ \int_{\mathbb{R}^{N}}\int_{\mathbb{R}^{N}}}\frac{\left\vert R_{m-1}f\left(
x+h,x\right)  \right\vert ^{p}}{\left\Vert h\right\Vert _{K}^{mp}}%
\rho_{\varepsilon}\left(  \left\Vert h\right\Vert _{K}\right)  dxdh
& \lesssim \underset{\left\vert h \right\vert \geq 1}
{\int_{\mathbb{R}^{N}}}\rho_{\varepsilon}\left(  \left\Vert
h\right\Vert _{K}\right)  dh\\
&  \lesssim\int_{1}^{\infty}\rho_{\varepsilon}\left(  s\right)
s^{N-1}ds\rightarrow0\text{ as }\varepsilon\rightarrow0.
\end{align*}
Letting $\varepsilon\rightarrow0$ and $\delta\rightarrow0$ we conclude,
\[
\limsup_{\varepsilon\rightarrow0}\int_{\mathbb{R}^{N}}\int_{\mathbb{R}^{N}}%
\frac{\left\vert R_{m-1}f\left(  x,y\right)  \right\vert ^{p}}{\left\Vert
x-y\right\Vert _{K}^{mp}}\rho_{\varepsilon}\left(  \left\Vert x-y\right\Vert
_{K}\right)  dxdy\leq\frac{N+mp}{\left(  m!\right)  ^{p}}
\int_{\mathbb{R}^{N}}\int_{K}\left\vert D^{m}f(x)(y,...,y)\right\vert
^{p}dydx.
\]
Assume $supp(f)\subset B_R$, then by Taylor's formula and using the fact that for any $\tau>0$ there exists $C\left(  \tau\right)  >1$ such that for
all $a,b\in%
\mathbb{R}
:$
\begin{align}\label{ineqmm}
\left\vert a\right\vert ^{p}\leq\left(  1+\tau\right)  \left\vert b\right\vert
^{p}+C\left(  \tau\right)  \left\vert a-b\right\vert ^{p}.
\end{align}
we get
\[
\frac{1}{\left(  m!\right)  ^{p}}\left\vert D^{m}f(x)(h,...,h)\right\vert
^{p}\leq\left(  1+\delta\right)  \left\vert R_{m-1}f\left(  x+h,x\right)
\right\vert ^{p}+C_{\delta,B}\left\vert h\right\vert ^{\left(  m+1\right)  p}.%
\]
Hence,%
\begin{align*}
&  \frac{1}{\left(  m!\right)  ^{p}}%
{\displaystyle\int\limits_{B}}
{\displaystyle\int\limits_{\left\vert h\right\vert \leq1}}
\frac{\left\vert D^{m}f(x)(h,...,h)\right\vert ^{p}}{\left\Vert h\right\Vert
_{K}^{mp}}\rho_{\varepsilon}\left(  \left\Vert h\right\Vert _{K}\right)
dxdh\\
&  \leq\left(  1+\delta\right)  \int_{\mathbb{R}^{N}}\int_{\mathbb{R}^{N}%
}\frac{\left\vert R_{m-1}f\left(  x+h,x\right)  \right\vert ^{p}}{\left\Vert
h\right\Vert _{K}^{mp}}\rho_{\varepsilon}\left(  \left\Vert h\right\Vert
_{K}\right)  dxdh+C_{\delta,B}%
{\displaystyle\int\limits_{B}}
{\displaystyle\int\limits_{\left\vert h\right\vert \leq1}}
\left\vert h\right\vert ^{\left(  m+1\right)  p}\frac{\rho_{\varepsilon
}\left(  \left\Vert h\right\Vert _{K}\right)  }{\left\Vert h\right\Vert
_{K}^{mp}}dxdh.
\end{align*}
Proceeding as in \eqref{kill} we get
\begin{align*}
&  \frac{1}{\left(  m!\right)  ^{p}}%
{\displaystyle\int\limits_{B}}
{\displaystyle\int\limits_{\left\vert h\right\vert \leq1}}
\left\vert D^{m}f(x)(\frac{h}{\left\vert h\right\vert },...,\frac
{h}{\left\vert h\right\vert })\right\vert ^{p}\left\vert h\right\vert
^{mp}\frac{\rho_{\varepsilon}\left(  \left\Vert h\right\Vert _{K}\right)
}{\left\Vert h\right\Vert _{K}^{mp}}dxdh\\
&=\frac{1}{\left(  m!\right)  ^{p}}%
{\displaystyle\int_{B}}
\int_{\mathbb{S}^{N-1}}\left\vert D^{m}f(x)(\sigma,...,\sigma)\right\vert
^{p}\frac{1}{\left\Vert \sigma\right\Vert _{K}^{N+mp}}d\sigma dx
\end{align*}
By letting $\varepsilon\rightarrow0$, and then $\delta\rightarrow0$, we
get%
\begin{align*}
& \frac{1}{\left(  m!\right)  ^{p}}%
{\displaystyle\int_{B}}
\int_{\mathbb{S}^{N-1}}\left\vert D^{m}f(x)(\sigma,...,\sigma)\right\vert
^{p}\frac{1}{\left\Vert \sigma\right\Vert _{K}^{N+mp}}d\sigma dx\\
&\leq
\liminf_{\varepsilon\rightarrow0}\int_{\mathbb{R}^{N}}\int_{\mathbb{R}^{N}}%
\frac{\left\vert R_{m-1}f\left(  x+h,x\right)  \right\vert ^{p}}{\left\Vert
h\right\Vert _{K}^{mp}}\rho_{\varepsilon}\left(  \left\Vert h\right\Vert
_{K}\right)  dxdh,
\end{align*}
this proves the thesis if $f\in C^{\infty}_c(\mathbb{R}^N)$.

In the general case $f\in W^{m,p}(\mathbb{R}^{N})$, we fix $\tau>0$ and let $C\left(  \tau\right)  >1$ be as in \eqref{ineqmm}. By density, we can choose $g\in C_{c}^{\infty}\left(
\mathbb{R}
^{N}\right)  $ such that%
\[
\int_{\mathbb{R}^{N}}|\nabla^{m}\left(  f-g\right)  \left(  x\right)
|^{p}dx\leq\frac{\tau}{C\left(  \tau\right)  }%
\]
and
\[
\left\vert \int_{\mathbb{R}^{N}}\int_{K}\left\vert D^{m}%
g(x)(y,...,y)\right\vert ^{p}dydx-\int_{\mathbb{R}^{N}}\int_{K}\left\vert
D^{m}f(x)(y,...,y)\right\vert ^{p}dydx\right\vert \leq\tau.
\]

Then
\begin{align*}
&  \int_{\mathbb{R}^{N}}\int_{\mathbb{R}^{N}}\frac{\left\vert R_{m-1}f\left(
x+h,x\right)  \right\vert ^{p}}{\left\Vert h\right\Vert _{K}^{mp}}%
\rho_{\varepsilon}\left(  \left\Vert h\right\Vert _{K}\right)  dxdh\\
&  \leq\left(  1+\tau\right)  \int_{\mathbb{R}^{N}}\int_{\mathbb{R}^{N}}%
\frac{\left\vert R_{m-1}g\left(  x+h,x\right)  \right\vert ^{p}}{\left\Vert
h\right\Vert _{K}^{mp}}\rho_{\varepsilon}\left(  \left\Vert h\right\Vert
_{K}\right)  dxdh\\
&  +C\left(  \tau\right)  \int_{\mathbb{R}^{N}}\int_{\mathbb{R}^{N}}%
\frac{\left\vert R_{m-1}\left(  f-g\right)  \left(  x+h,x\right)  \right\vert
^{p}}{\left\Vert h\right\Vert _{K}^{mp}}\rho_{\varepsilon}\left(  \left\Vert
h\right\Vert _{K}\right)  dxdh\\
&  \leq\left(  1+\tau\right)  \int_{\mathbb{R}^{N}}\int_{\mathbb{R}^{N}}%
\frac{\left\vert R_{m-1}g\left(  x+h,x\right)  \right\vert ^{p}}{\left\Vert
h\right\Vert _{K}^{mp}}\rho_{\varepsilon}\left(  \left\Vert h\right\Vert
_{K}\right)  dxdh\\
&  +C\left(  \tau\right)  C(m,N,p)\int_{\mathbb{R}^{N}}|\nabla^{m}\left(
f-g\right)  \left(  x\right)  |^{p}dx.
\end{align*}
Letting $\varepsilon\rightarrow0$, we obtain%
\begin{align*}
&  \limsup_{\varepsilon\rightarrow0}\int_{\mathbb{R}^{N}}\int_{\mathbb{R}^{N}%
}\frac{\left\vert R_{m-1}f\left(  x+h,x\right)  \right\vert ^{p}}{\left\Vert
h\right\Vert _{K}^{mp}}\rho_{\varepsilon}\left(  \left\Vert h\right\Vert
_{K}\right)  dxdh\\
&  \leq\left(  1+\tau\right)  \frac{N+mp}{\left(  m!\right)  ^{p}}\int_{\mathbb{R}^{N}}\int_{K}\left\vert D^{m}g(x)(y,...,y)\right\vert
^{p}dydx+C(m,N,p)\tau\\
&  \leq\left(  1+\tau\right)  \frac{N+mp}{\left(  m!\right)  ^{p}}\left[  \int_{\mathbb{R}^{N}}\int_{K}\left\vert D^{m}%
f(x)(y,...,y)\right\vert ^{p}dydx+\tau\right]  +C(m,N,p)\tau.
\end{align*}
Since $\tau$ can be chosen arbitrarily, we deduce that
\begin{align*}
&  \limsup_{\varepsilon\rightarrow0}\int_{\mathbb{R}^{N}}\int_{\mathbb{R}^{N}%
}\frac{\left\vert R_{m-1}f\left(  x+h,x\right)  \right\vert ^{p}}{\left\Vert
h\right\Vert _{K}^{mp}}\rho_{\varepsilon}\left(  \left\Vert h\right\Vert
_{K}\right)  dxdh\\
&  \leq\frac{N+mp}{\left(  m!\right)  ^{p}}\int_{\mathbb{R}^{N}%
}\int_{K}\left\vert D^{m}f(x)(y,...,y)\right\vert ^{p}dydx.
\end{align*}
Also, if we switch the role of $f$ and $g$ in the above argument, then we get%
\begin{align*}
&  \frac{N+mp}{\left(  m!\right)  ^{p}}\int_{\mathbb{R}^{N}}%
\int_{K}\left\vert D^{m}f(x)(y,...,y)\right\vert ^{p}dydx\\
&  \leq\liminf_{\varepsilon\rightarrow0}\int_{\mathbb{R}^{N}}\int_{\mathbb{R}%
^{N}}\frac{\left\vert R_{m-1}f\left(  x+h,x\right)  \right\vert ^{p}%
}{\left\Vert h\right\Vert _{K}^{mp}}\rho_{\varepsilon}\left(  \left\Vert
h\right\Vert _{K}\right)  dxdh.
\end{align*}
Hence,%
\begin{align*}
&\lim_{\varepsilon\rightarrow0}\int_{\mathbb{R}^{N}}\int_{\mathbb{R}^{N}}%
\frac{\left\vert R_{m-1}f\left(  x+h,x\right)  \right\vert ^{p}}{\left\Vert
h\right\Vert _{K}^{mp}}\rho_{\varepsilon}\left(  \left\Vert h\right\Vert
_{K}\right)  dxdh\\
&=\frac{N+mp}{\left(  m!\right)  ^{p}}
\int_{\mathbb{R}^{N}}\int_{K}\left\vert D^{m}f(x)(y,...,y)\right\vert
^{p}dydx.
\end{align*}

\end{proof}

\textbf{Acknowledgements.}
The authors would like to thank Quoc-Hung Nguyen and Professor Hoai-Minh Nguyen for their interest in our work and for stimulating discussions during the preparation of the manuscript. A.P. is member of the Gruppo Nazionale per l’Analisi Matematica, la Probabilità e le loro Applicazioni (GNAMPA) of the Istituto Nazionale di Alta Matematica (INdAM).


\begin{thebibliography}{99}                                                                                               %


\bibitem {AmbDepMart}L.\ Ambrosio, G.\ De Philippis, L.\ Martinazzi,
\textit{$\Gamma-$convergence of nonlocal perimeter functionals,} Manuscripta
Math. \textbf{134} (2011), 377--403.



%
\bibitem{BIK2011} B.\ Bojarski, L.\ Ihnatsyeva, J.\ Kinnunen, \textit{How to recognize polynomials in higher order Sobolev spaces}, Math. Scand. 112 (2013), no. 2, 161--181.

\bibitem{B}R.\,Borghol, \textit{Some properties of Sobolev spaces},
Asymptotic Analysis 51 (2007) 303--318.
\bibitem{Barb} D.\ Barbieri, \textit{Approximations of Sobolev norms in Carnot groups}, Comm.\ Contemp.\ Math.\ {\bf13} (2011), 765--794.
\bibitem {BBN23}J.\ Bourgain, H.\ Brezis, H-M.\ Nguyen, \textit{A new estimate for the topological degree}, C.R.Acad.Sci.Paris \textbf{343} (2006), 75--80.

\bibitem {bourg}J. Bourgain, H. Brezis, P. Mironescu, \textit{Another look at
Sobolev spaces}, in \emph{Optimal Control and Partial Differential Equations.
A Volume in Honor of Professor Alain Bensoussan's 60th Birthday} (eds. J. L.
Menaldi, E. Rofman and A. Sulem), IOS Press, Amsterdam, 2001, 439--455.

\bibitem {bourg2}J. Bourgain, H. Brezis, P. Mironescu, \textit{Limiting
embedding theorems for $W^{s,p}$ when $s \uparrow1$ and applications}, {J.
Anal. Math.} \textbf{87} (2002), 77--101.

%


\bibitem {bre}H. Brezis, \textit{How to recognize constant functions.
Connections with Sobolev spaces}, Russian Mathematical Surveys \textbf{57}
(2002), 693--708.

\bibitem {bre-linc}H. Brezis, \textit{New approximations of the total
variation and filters in imaging}, Rend Accad. Lincei \textbf{26} (2015), 223--240.





\bibitem {BHN3}H. Brezis, H.-M. Nguyen, \textit{Two subtle convex nonlocal
approximations of the BV-norm}, Nonlinear Anal. \textbf{137} (2016), 222--245.

%



\bibitem {Cui}X.\ Cui, N.\ Lam, G.\ Lu, \textit{New characterizations of
Sobolev spaces in the Heisenberg group}, J. Funct. Anal. \textbf{267} (2014), 2962--2994.

\bibitem {Cui2}X.\ Cui, N.\ Lam, G.\ Lu, \textit{Characterizations of second
order Sobolev spaces}, Nonlinear Anal. 121 (2015), 241--261.


\bibitem{MSqu} S.\ Di Marino, M.\ Squassina, \textit{New characterizations of Sobolev metric spaces}, preprint (2018). Available at \href{http://cvgmt.sns.it/paper/3791/}{http://cvgmt.sns.it/paper/3791/}.








%
\bibitem{KL}
A. Kreuml, O. Mordhorst, {\it Fractional Sobolev norms and BV functions on manifolds}, preprint (2018). Available at \href{https://arxiv.org/abs/1805.04425}{arXiv:1805.04425}.

\bibitem{Leoni}
G. Leoni, {\it A first course in Sobolev Spaces, } Second Edition, Graduate Studies in Mathematics. 181. Providence, RI: American Mathematical Society (AMS), 2017.



%


%


\bibitem {nguyen06}H.-M.\ Nguyen, \textit{Some new characterizations of
Sobolev spaces}, J. Funct. Anal. \textbf{237} (2006), 689--720.

\bibitem {nguyen07}H.-M.\ Nguyen, \textit{Some inequalities related to Sobolev
norms}, Calculus of Variations and Partial Differential Equations 41 (2011) 483--509.

\bibitem {NguPinSquVec}H.-M.\ Nguyen, A.\ Pinamonti, M.\ Squassina,
E.\ Vecchi, \textit{New characterization of magnetic Sobolev spaces},  Adv. Nonlinear Anal. 7 (2018), no. 2, 227–245.

\bibitem {NguPinSquVec2}H.-M.\ Nguyen, A.\ Pinamonti, M.\ Squassina,
E.\ Vecchi, \textit{Some characterizations of magnetic Sobolev spaces}, to appear in Complex Variables and Elliptic Equations. DOI:10.1080/17476933.2018.1520850

\bibitem {NguSq}H-M.\ Nguyen, M.\ Squassina, \textit{On anisotropic Sobolev spaces},
Commun. Contemp. Math, to appear.

\bibitem {PinSquVec}A.\ Pinamonti, M.\ Squassina, E.\ Vecchi, \textit{Magnetic BV
functions and the Bourgain-Brezis-Mironescu formula}, Adv. Calc. Var. (2017), https://doi.org/10.1515/acv-2017-0019

\bibitem {PinSquVec2}A.\ Pinamonti, M.\ Squassina, E.\ Vecchi, \textit{The
Maz'ya-Shaposhnikova limit in the magnetic setting}, J. Math. Anal. Appl.,
\textbf{449} (2017), 1152-1159.

\bibitem {Ponce}A.\ Ponce, \textit{A new approach to Sobolev spaces and
connections to $\Gamma$-convergence,} Calc. Var. Partial Differential
Equations \textbf{19} (2004), 229--255.



\bibitem {BM}M.\ Squassina, B.\ Volzone, \textit{Bourgain-Brezis-Mironescu
formula for magnetic operators,} C. R. Math. Acad. Sci. Paris \textbf{354}
(2016), 825--831.

\bibitem {Stein}E.\ M.\ Stein, \textit{Singular integrals and
differentiability properties of functions}. Princeton University Press,
Princeton, N.J., 1970.

\end{thebibliography}
\end{document}